\documentclass[12pt]{article}
\usepackage{amsmath}
\usepackage{amsthm}
\usepackage{amsfonts}
\usepackage{float}
\usepackage{diagbox}
\usepackage[dvips,final]{graphicx}
\usepackage{dcolumn}
\usepackage{comment} 
\usepackage{xspace}
\usepackage{url}
\usepackage{upgreek}
\usepackage{multicol}
\usepackage{pict2e,multirow}
\usepackage{verbatim}
\usepackage{wrapfig}
\usepackage{url}
\usepackage{tikz}
\usepackage[colorlinks=true,citecolor=blue,urlcolor=blue]{hyperref}
\usepackage{subfigure}
\usepackage[T1]{fontenc}
\usepackage[hmargin=.720in]{geometry}

\newtheorem{thm}{Theorem}

\newtheorem{lem}{Lemma}
\newtheorem{cor}{Corollary}

\newcommand{\be}{\begin{equation}}
\newcommand{\ee}{\end{equation}}

\newcommand{\Dt}{\Delta t}

\newcommand{\dx}{\Delta x}
\newcommand{\dt}{\Delta t}
\newcommand{\aij}{\alpha_{i,j}}
\newcommand{\bij}{\beta_{i,j}}

\newcommand{\sspcoef}{\mathcal{C}}

\newcommand{\DtFE}{\Dt_{\textup{FE}}}
\newcommand{\tDtFE}{\tilde{\Dt}_{\textup{FE}}}

\newcommand{\ft}{\tilde{F}}

\title{Downwinding for Preserving Strong Stability in Explicit Integrating Factor Runge--Kutta Methods}

\author{%
Leah Isherwood\thanks{Mathematics Department, University of Massachusetts Dartmouth, 285 Old Westport Road,
North Dartmouth MA 02747.},
Zachary J. Grant\footnotemark[1], and 
Sigal Gottlieb\footnotemark[1]
}
\begin{document}
\maketitle

%\begin{keywords}
%\end{keywords}

\bibliographystyle{siam}

\begin{abstract} 
Strong stability preserving (SSP)  Runge--Kutta methods are  desirable  when evolving in time  
problems that have discontinuities or sharp gradients and require nonlinear non-inner-product 
stability properties to be satisfied. Unlike the case for $L_2$ linear stability, implicit methods do
not significantly alleviate the time-step restriction when the SSP property is needed. For this
reason, when handling problems with  a linear component that is stiff and a  nonlinear component that is not,
SSP integrating factor Runge--Kutta methods may offer an attractive alternative to traditional time-stepping methods.
The strong stability properties of integrating factor Runge--Kutta methods where the 
transformed problem is evolved with an explicit SSP Runge--Kutta method with non-decreasing abscissas
was recently established. However, these methods typically have smaller SSP coefficients (and therefore
a smaller allowable time-step) than the optimal SSP Runge--Kutta methods, which often have
some decreasing abscissas. In this work, we consider the use of downwinded spatial operators to
preserve the strong stability properties of integrating factor Runge--Kutta methods  where the
Runge--Kutta method has some decreasing abscissas. We present the SSP theory for this approach 
and present numerical evidence to show that such an approach is feasible and performs as expected.
However, we also show that  in some cases the integrating factor approach
with explicit SSP Runge--Kutta methods with non-decreasing abscissas performs nearly as well, 
 if not better, than with explicit SSP Runge--Kutta methods with downwinding.
In conclusion, while the downwinding approach can be rigorously shown to guarantee the  SSP property for a larger time-step,
 in practice  using the integrating factor approach
by including downwinding as needed with   optimal explicit SSP Runge--Kutta methods does not 
 necessarily provide significant benefit over using  explicit SSP Runge--Kutta methods with non-decreasing abscissas.
 \end{abstract}

\begin{center}
{\bf This paper is in honor of Prof. Chi-Wang Shu's sixtieth birthday. \\
His pioneering work on SSP methods and
his observations on downwinding inspired this paper. We wish him many many more productive, happy, and healthy years
to inspire many mathematicians.}
\end{center}

\section{Introduction\label{sec:intro}}
When numerically solving a hyperbolic conservation law of the form 
\begin{eqnarray} \label{PDE}
U_t +f(U)_x = 0,
\end{eqnarray}
specially designed spatial discretizations are used to handle the discontinuities in 
the solution that sometimes  arise. These spatial discretizations typically satisfy some 
nonlinear non-inner-product strong stability properties when coupled with forward Euler time-stepping 
\cite{SSPbook2011}. However, in practice we wish to use higher order time discretizations, which preserve the strong stability 
properties of the spatial discretization coupled with forward Euler.

Explicit strong stability preserving (SSP) Runge--Kutta methods were  first developed in \cite{shu1988b,shu1988}
to evolve the semi-discretization
\begin{eqnarray}
\label{ode}
u_t = F(u),
\end{eqnarray}
resulting from approximating $f(u)_x$ 
 with a total variation diminishing (TVD) spatial discretization.
TVD spatial discretizations are specially designed to 
ensure that the forward Euler time-step is strongly stable  
\begin{eqnarray} \label{FEstrongstability}
\|u^{n+1} \| = \| u^n + \dt F(u^{n}) \| \leq \| u^n \| 
\end{eqnarray}
under some  step size restriction
\begin{eqnarray} \label{FEcond}
0 \leq \dt \leq \DtFE.
\end{eqnarray}
We wish to guarantee  that the same type of  strong stability property 
\begin{eqnarray} \label{monotonicity}
\|u^{n+1} \| \leq \|u^{n} \| 
\end{eqnarray}
is still satisfied when the TVD spatial discretization is coupled with a higher order time-stepping method. To do this,  we 
use the fact that many higher order time discretization can be written as a convex combination of forward Euler steps.

It is simple to show that if we can  re-write a higher order time discretization as a convex combination of forward Euler steps,
then we can ensure that any convex functional property \eqref{monotonicity} that is satisfied 
by the forward Euler method will still be satisfied by the  higher order time discretization, perhaps under a different time-step.
For example, an  $s$-stage explicit Runge--Kutta method can be written as:
\begin{eqnarray}
\label{rkSO}
u^{(0)} & =  & u^n, \nonumber \\
u^{(i)} & = & \sum_{j=0}^{i-1} \left( \aij u^{(j)} +
\dt \bij F(u^{(j)}) \right), \; \; \; \; i=1, . . ., s\\% \quad \aik \geq 0,  \qquad i=1 ,..., m \\
 u^{n+1} & = & u^{(s)} . \nonumber
\end{eqnarray}
Each stage can be written as
\[ u^{(i)}  =  \sum_{j=0}^{i-1} \aij \left(  u^{(j)} + \dt \frac{\bij}{\aij}  F(u^{(j)}) \right)  \]
provided that  a given $\aij$ is zero only if its corresponding $\bij$ is zero. Recall that 
for consistency, we must have $\sum_{j=0}^{i-1} \aij =1$, so that as long as the
coefficients $\aij$ and $\bij$ are all non-negative, 
each stage can be rearranged into a convex combination of forward Euler steps.
Thus we have 
\[
\| u^{(i)}\|  =   \left\| \sum_{j=0}^{i-1} \left( \aij u^{(j)} + \dt \bij F(u^{(j)}) \right) \right\|   \nonumber 
 \leq   \sum_{j=0}^{i-1} \aij  \, \left\| u^{(j)} + \dt \frac{\bij}{\aij} F(u^{(j}) \right\|  
 \leq  \|u^n\| , \]
(where the final inequality follows from  \eqref{FEstrongstability} and  \eqref{FEcond}),
for any time-step that satisfies
\begin{eqnarray}
\dt \leq \min_{i,j} \frac{\aij}{\bij} \DtFE.
\end{eqnarray}
If any of the $\beta$'s are equal to zero,  we consider the corresponding ratio to be infinite.

In the case where a particular $\bij<0$,  the SSP property can still be guaranteed 
provided that we modify the spatial discretization for these instances  \cite{shu1988}.
When $\bij$ is negative, $\bij F(u^{(k)})$ is replaced by $\bij \ft(u^{(k)})$,
where $\ft$ approximates the same spatial derivative(s) as $F$, but the 
strong stability property
$\|u^{n+1}\| \leq \|u^n\|$
 holds  for the first order Euler scheme, solved backward in time, i.e.,
\begin{eqnarray}
\label{1.10}
u^{n+1} = u^n - \dt \ft (u^n)
\end{eqnarray}
This can be achieved, for hyperbolic conservation laws, by solving the
negative in time version of \eqref{PDE},
\begin{eqnarray*}
U_t -  f(U)_x = 0 . 
\end{eqnarray*}
Numerically, the only difference is the change of the upwind direction.
Thus, if $\aij \geq 0$, all the intermediate stages $u^{(i)}$ in \eqref{rkSO}
are  convex combinations of
backward in time Euler  and forward Euler operators, 
with $\dt$ replaced by $\frac{|\bij|}{\aij} \dt$. Following the same reasoning as above, any strong
stability bound satisfied by the backward in time and forward in time
Euler methods will then be preserved by the Runge--Kutta method \eqref{rkSO}
where $F$ is replaced by $\ft$ whenever the corresponding $\beta$ is negative.

%The last inequality above  follows from  the strong stability conditions \eqref{FEstrongstability} and  \eqref{FEcond} 
%\[ \left\| u^{(j)} + \dt F(u^{(j})  \right\|   \leq \left\| u^{(j)}   \right\|   \; \; \; \forall \dt \leq \DtFE \]  and 
%the consistency condition $\sum_{j=0}^{i-1} \aij =1$. 
Clearly then, if we can re-write an explicit Runge--Kutta method as a  convex combination
of forward Euler steps (or, in the downwinded case, of forward Euler and backward in time Euler steps), 
the monotonicity condition  \eqref{FEstrongstability}  will be {\em preserved} by the higher-order time 
discretizations, under a modified   time-step restriction 
$\Dt \le \sspcoef \DtFE$ 
where  $\sspcoef = \min_{i,j} \frac{\aij}{|\bij|}$.   As long as $ \sspcoef  >0$, the method is 
called {\em strong stability preserving} (SSP) with {\em SSP coefficient}  \;  $\sspcoef$ \cite{shu1988b}.  
Methods that use the downwinded  operator $\ft$ as well as the operator $F$ are
 called downwinded methods \cite{SSPbook2011}.
 
In the original papers, the  term $\| \cdot \|$ in Equation  \eqref{FEstrongstability} above  represented
the total variation semi-norm, and these methods were known as TVD time-stepping methods
  \cite{shu1988b, shu1988}.
However, the strong stability  preservation property holds
for any semi-norm, norm, or convex functional, as determined by the design of the spatial discretization, provided {\em only}
that the forward Euler condition   \eqref{FEstrongstability}  holds, and that the time-discretization 
can be decomposed into a convex combination of forward Euler and backward in time Euler steps  with   $ \sspcoef  >0$.

The convex combination condition is not only a sufficient condition for strong stability preservation,
 it is also necessary for  strong stability preservation
 \cite{SSPbook2011, kraaijevanger1991,spijker2007}. This means that 
if a method cannot be decomposed into a convex combination of forward Euler steps, then  we can always find 
some ODE with some initial condition such that the forward Euler condition is satisfied but the method
does not satisfy the strong stability condition for any positive time-step \cite{SSPbook2011}.
%Methods that can be decomposed like this with with  $\sspcoef > 0$ are 
%called strong stability preserving (SSP), 
%and the coefficient $\sspcoef$ is known as the {\em SSP coefficient} of the method.

Not every method can be decomposed into convex combinations of forward Euler steps with
$\sspcoef >0$. For this reason, explicit SSP Runge--Kutta methods  
cannot exist for order $p>4$   \cite{kraaijevanger1991,ruuth2001}.
% it was shown that  Furthermore, the value of $\sspcoef$ determines in large part what
%the size of an allowable time-step will be, and so we seek methods that  have the largest possible
%SSP coefficient. Of course, a more important quantity is the total cost of the time evolution,
% which is related to the allowable time step  relative to the number of function evaluations
%at each time-step.  To allow us to compare the efficiency of explicit methods of a given order,
%we  define the {\em effective SSP coefficient} $\ceff = \frac{\sspcoef}{s}$
%where $s$ is the number of stages (typically the number of function evaluations). 
Furthermore, the SSP requirement is quite restrictive, so that  all explicit $s$-stage Runge--Kutta 
methods have an SSP bound $\sspcoef \leq s$ \cite{SSPbook2011}. Moreover,
this upper bound cannot usually be attained. Nevertheless, many efficient explicit SSP Runge--Kutta methods 
exist and are discussed in Section \ref{sec:background}.
Implicit  SSP Runge--Kutta methods have been an active area of investigation as well; these methods have an
order barrier of $p \leq 6$, and seem to exhibit an SSP bound $\sspcoef \leq 2 s$ \cite{SSPbook2011}. 
This disappointing result limits the interest in implicit SSP Runge--Kutta methods, as well as in
implicit-explicit SSP Runge--Kutta methods, studied in \cite{IMEX}.
 
Given a semi-discretized problem of the form
 \[ u_t = Lu + N(u) \]
 where $L$ is a linear operator that significantly restricts the time-step, we typically turn to 
implicit-explicit methods to alleviate the time-step restriction. However, when the 
 time-step is restricted  due to nonlinear non-inner-product stability considerations, SSP methods are necessary,
 but  implicit-explicit SSP Runge--Kutta methods do not significantly alleviate the time-step restriction \cite{IMEX}.
This motivated our initial  investigation into integrating factor methods \cite{SSPIFRK2018}, where the 
 linear component $L u$ is handled exactly, and then 
 the allowable time-step depends only upon the nonlinear component $N(u)$. 
 In \cite{SSPIFRK2018} we discussed the conditions under which 
this process  guarantees that the strong stability property \eqref{monotonicity} is preserved. 
In that work, we showed that if we  step the transformed problem forward using 
an SSP Runge--Kutta method where
the abscissas (i.e. the time-levels approximated by each stage) are non-decreasing, 
we obtain a method that preserves the desired  strong stability property.
 These non-decreasing
abscissa SSP Runge--Kutta methods usually have smaller SSP coefficients than the
optimal explicit  SSP Runge--Kutta methods. 
However, there is an alternative approach inspired by classical SSP theory: for the stages where the abscissas are 
decreasing, we can replace the operator $L$ in the exponential with the downwind operator $\tilde{L}$, and 
the resulting method will be SSP with the original SSP time-step.

In the current work we discuss the downwinding approach in the context of  integrating factor Runge--Kutta methods.
In our case, the Runge--Kutta method does not have negative coefficients, but  some stages  the
difference of abscissas  is  negative (i.e.  some of the abscissas are decreasing). 
To preserve the SSP property we can
replace the operator $L$ with the downwind operator $\tilde{L}$ for cases where the abscissas are decreasing.
The extra cost of computing the exponential for $\tilde{L}$ can be significant if needed at each time-step; however, if
the exponential operators for both $L$ and $\tilde{L}$ are pre-computed, the additional cost is negligible.
In this paper we rigorously prove this approach to be SSP  and show how it works
on simple test cases. Our conclusions are that while this approach is viable, it is not necessarily more
efficient than  the integrating factor approach using the non-descreasing abscissa Runge--Kutta methods described in 
 \cite{SSPIFRK2018}, particularly if the exponential operators are not pre-computed.

In Section  \ref{sec:SSPIF} we provide the SSP theory for integrating factor Runge--Kutta methods.
In Section \ref{sec:background} we review the optimal explicit SSP Runge--Kutta methods that serve as a basis for the
SSP integrating factor Runge--Kutta methods, and provide their SSP coefficients. 
Next, in Section \ref{sec:test} we demonstrate through numerical examples the need for downwinding
in the case where the explicit Runge--Kutta method has some decreasing abscissas, and compare the
use of downwinding  to the  non-decreasing abscissa approach.
We also show that although including downwinding changes the ODE, so that time-refinement alone will not
show convergence,  refinement in both space and time will show  convergence to the solution of the PDE.
We conclude that downwinding is a numerically viable approach that can be rigorously shown to preserve the 
strong stability properties when used with an integrating factor Runge--Kutta approach, but may not be
more beneficial than using the integrating factor approach with Runge--Kutta methods that have only non-decreasing abscissas.

\section{SSP theory for explicit integrating factor Runge--Kutta  methods} \label{sec:SSPIF} 
We consider a hyperbolic PDE whose semi-discretization results in an ODE system of the form
\begin{eqnarray}
u_t = Lu + N(u)  
\end{eqnarray}
with a nonlinear component $N(u)$ that satisfies 
\begin{eqnarray}  \label{nonlinearFEcond}
\| u^n + \dt N(u^n) \| \leq \|u^n\| \qquad \mbox{for} \qquad \dt \leq \DtFE
\end{eqnarray}
and a linear constant coefficient component $L u $ that satisfies 
\begin{eqnarray} \label{linearFEcond}
\| u^n + \dt L u^n \| \leq \|u^n\| \qquad \mbox{for} \qquad \dt \leq \tDtFE
\end{eqnarray}
for some convex functional  $\| \cdot \|$. 
In this case, the allowable time-step for the linear component is significantly smaller than the one
for the nonlinear component, $ \tDtFE <<  \DtFE$.   
In such cases, stepping forward using an explicit SSP Runge--Kutta method, or even an implicit-explicit (IMEX)
SSP Runge--Kutta method will result in severe constraints on the allowable time-step. We seek a time-stepping approach
that alleviates the time-step restriction while preserving the  monotonicity property $\| u^{n+1} \| \leq \| u^n \|$.

As in \cite{SSPIFRK2018}  we wish to treat the  linear part exactly using an integrating factor approach 
\[  e^{-L t } u_t - e^{-L t }  Lu = e^{-L t } N(u)  \longrightarrow
\left( e^{-L t } u \right)_t = e^{-L t } N(u). \]
Defining $w = e^{-L t } u $ gives the ODE system
\begin{eqnarray} \label{IF_ODE}
w_t  = e^{-L t } N(e^{L t } w)   = G(w),
\end{eqnarray}
which we then evolve in time using an explicit Runge--Kutta method of the form \eqref{rkSO}.
This approach is known as a Lawson-type method \cite{lawson1967}.
%\begin{eqnarray} \label{
%w^{(0)} & =  & w^n, \nonumber \\
%w^{(i)} & = & \sum_{j=0}^{i-1} \left( \aij w^{(j)} + \dt \bij G(w^{(j)}) \right), \; \; \; \; i=1, . . ., s\\
% w^{n+1} & = & w^{(s)}  \nonumber
%\end{eqnarray}
%by using it to integrate the ODE system \eqref{IF_ODE}. 

Each stage $u^{(i)}$ of \eqref{rkSO} becomes
\[
e^{-L t_i} u^{(i)}   =   \sum_{j=0}^{i-1} \left( \aij e^{-L t_j}  u^{(j)} + \dt \bij e^{-L t_j}  N(u^{(j)}) \right),  \]
or
\begin{eqnarray}
u^{(i)}   & =  &   \sum_{j=0}^{i-1} \left( \aij e^{L (t_i-t_j) }  u^{(j)} + \dt \bij e^{L ( t_i-t_j)}  N(u^{(j)}) \right)  \\
& =  &   \sum_{j=0}^{i-1} \left( \aij e^{L (c_i-c_j) \dt }  u^{(j)} + \dt \bij e^{L ( c_i-c_j)\dt}  N(u^{(j)}) \right) . 
\end{eqnarray}
This stage corresponds to the solution at time $t_i = t^n + c_i \dt$,
where each $c_i$ is the abscissa of the method at the $i$th stage.

In our prior work, we used the two properties \eqref{nonlinearFEcond} and \eqref{linearFEcond} 
to establish the SSP properties of an integrating factor Runge--Kutta method
in the case where the abscissas are non-decreasing. In this work, we wish to allow decreasing abscissas in order to 
enlarge the SSP coefficient. For this purpose, 
we also define the downwinded operator $\tilde{L}$ which approximates the same term in the PDE as
$L$, but satisfies the strong stability condition:
\begin{eqnarray} \label{linearFEcond2}
\| u^n - \dt \tilde{L} u^n \| \leq \|u^n\| \qquad \mbox{for} \qquad \dt \leq \tDtFE.
\end{eqnarray}
For hyperbolic partial differential equations, this is accomplished by using the spatial discretization that is stable for 
a downwind problem. This approach is similar to the one employed  in the classical SSP literature, where negative coefficients $\bij$ may be allowed if the 
corresponding operator is replaced by a downwinded operator. However, in our case all the coefficients of the Runge--Kutta methods are nonnegative,
and the negative terms appear only in the exponential, due to decreasing abscissas.

\begin{thm} \label{thm:exp_FEcond}
{\em (From \cite{SSPIFRK2018})}
If  a  linear operator $L$  satisfies  \eqref{linearFEcond} for some value of $\tDtFE > 0 $,
then 
\begin{equation} \label{EXPcondition}
\| e^{\tau L} u^n \|  \leq \| u^n \|   \; \; \; \; \forall \; \tau \geq 0 . 
\end{equation}
\end{thm}
This theorem was proved in  \cite{SSPIFRK2018}. Clearly, if we simply replace $ L$ with $- \tilde{L}$, and the corresponding 
condition \eqref{linearFEcond} with \eqref{linearFEcond2} we obtain a
similar result for the downwinded operator:

\begin{cor} \label{thm:exp_FEcond2}
If  a  linear operator $\tilde{L}$  satisfies  \eqref{linearFEcond2} for some value of $\tDtFE > 0 $,
then 
\begin{equation} \label{EXPcondition2}
\| e^{-\tau \tilde{L}} u^n \|  \leq \| u^n \|   \; \; \; \; \forall \; \tau \geq 0 . 
\end{equation}
\end{cor}

\begin{lem} \label{stab}
{\em (From \cite{SSPIFRK2018})}
Given a  linear operator $L$ that satisfies  \eqref{EXPcondition}
and a (possibly nonlinear) operator $N(u)$ that satisfies \eqref{nonlinearFEcond}  
for some value of ${\Delta t}_{FE} \geq 0 $, we have 
\begin{equation}
\| e^{\tau L} ( u^n  + \dt N (u^n )) \|  \leq \| u^n \|   \; \; \; \; \forall \dt \leq \DtFE,  \; \; \; \mbox{provided that} \; \; \tau \geq 0. 
\end{equation}
\end{lem}
This Lemma  was also proved in  \cite{SSPIFRK2018}. Once again,  simply replacing $ L$ with $- \tilde{L}$, and the corresponding 
condition \eqref{EXPcondition} with \eqref{EXPcondition2} we obtain a
similar result for the downwinded operator:

\begin{cor} \label{DWstab}
Given a  linear operator $\tilde{L}$ that satisfies  \eqref{EXPcondition2}
and a (possibly nonlinear) operator $N(u)$ that satisfies \eqref{nonlinearFEcond}  
for some value of ${\Delta t}_{FE} \geq 0 $, we have 
\begin{equation}
\| e^{- \tau \tilde{L}} ( u^n  + \dt N (u^n )) \|  \leq \| u^n \|   \; \; \; \; \forall \dt \leq \DtFE,  \; \; \; \mbox{provided that} \; \; \tau \geq 0. 
\end{equation}
\end{cor}

The following theorem establishes the conditions under which an integrating factor Runge--Kutta method
which incorporates the downwinded operator $\tilde{L}$  is strong stability preserving:
\begin{thm} \label{thm:SSPIF}
Given   linear operators $L$ and $\tilde{L}$ that satisfy   \eqref{EXPcondition}
and  \eqref{EXPcondition2}, respectively, 
 a (possibly nonlinear) operator $N(u)$ that satisfies \eqref{nonlinearFEcond}  
for some value of ${\Delta t}_{FE} > 0 $, and a Runge--Kutta integrating factor method of the form
\begin{eqnarray} \label{rkIFSO}
u^{(0)} & =  & u^n, \nonumber \\
u^{(i)} & = & \sum_{j=0}^{i-1} e^{L_{ij}^* (c_i-c_j) \dt }  \left( \aij u^{(j)} + \dt \bij N(u^{(j)}) \right), \; \; \; \; i=1, . . ., s\\
 u^{n+1} & = & u^{(s)}  \nonumber
\end{eqnarray}
where $L_{ij}^* = L$ when $c_i \geq c_j$, and $L_{ij}^* = \tilde{L}$ when $c_i < c_j$,
 then $u^{n+1}$ obtained from \eqref{rkIFSO} satisfies
\begin{equation}
\|u^{n+1}\| \leq \|u^n\| \; \; \; \forall \dt \leq \sspcoef \DtFE.
\end{equation}
where 
\[ \sspcoef  =  \min_{i,j} \frac{\aij}{\bij}.\]
\end{thm}
\begin{proof}
We observe that  for each stage of \eqref{rkIFSO}
\begin{eqnarray*}
\| u^{(i)} \|& = & \left\| \sum_{j=0}^{i-1} e^{L_{ij}^* (c_i-c_j) \dt }  \left( \aij u^{(j)} + \dt \bij N(u^{(j)}) \right) \right\| \\
& \leq  & \sum_{j=0}^{i-1} \left\|  e^{L_{ij}^* (c_i-c_j) \dt }  \left( \aij u^{(j)} + \dt \bij N(u^{(j)}) \right)\right\| \\
& \leq  &  \sum_{j=0}^{i-1} \aij \left\|  e^{L_{ij}^* (c_i-c_j) \dt }  \left( u^{(j)} + \dt \frac{\bij}{\aij} N(u^{(j)}) \right) \right\| 
\end{eqnarray*}
where the last inequality follows from Lemma \ref{stab} and Corrolary \ref{DWstab}.
%Furthermore, this proof ensures that these methods have internal stage strong stability as well,  i.e.
%$\| u^{(i+1)} \| \leq \| u^{(i)} \| $ at each stage $i$ of the time-stepping, under the same time-step restriction.
\end{proof}

\bigskip
The following example demonstrates the need for using the downwind operator when the abscissas are decreasing.

\noindent{\bf Motivating Example:} To demonstrate the practical importance of this theorem, consider the partial differential equation
 \begin{align} 
U_t + a U_x +  \left( \frac{1}{2} U^2 \right)_x & = 0 \hspace{.75in}
    u(0,x)  =
\begin{cases}
1, & \text{if } 0 \leq x \leq 1/2 \\
0, & \text{if } x>1/2 \nonumber
\end{cases}
\end{align}
on the domain $[0,1]$ with periodic boundary conditions. We discretize the spatial grid with $400$ points and use
a first-order upwind difference  $L u \approx -a u_x  $ for $a>0$ defined by 
\begin{eqnarray} \label{upwinding}
 (L u)_j =  -a \left(  \frac{u_{j} - u_{j-1} }{\Delta x} \right) 
 \end{eqnarray} 
to semi-discretize the linear term. This operator satisfies the TVD condition 
\[ \| u^n +  \dt    {L} u \|_{TV} \leq \| u^n \|_{TV}  \; \; \; \; \mbox{for} \; \; \; \dt \leq \frac{1}{a} \dx.\] 
In this example, we use $a=10$.

\begin{figure}
  \begin{minipage}[c]{0.57\textwidth}
\includegraphics[scale=.55]{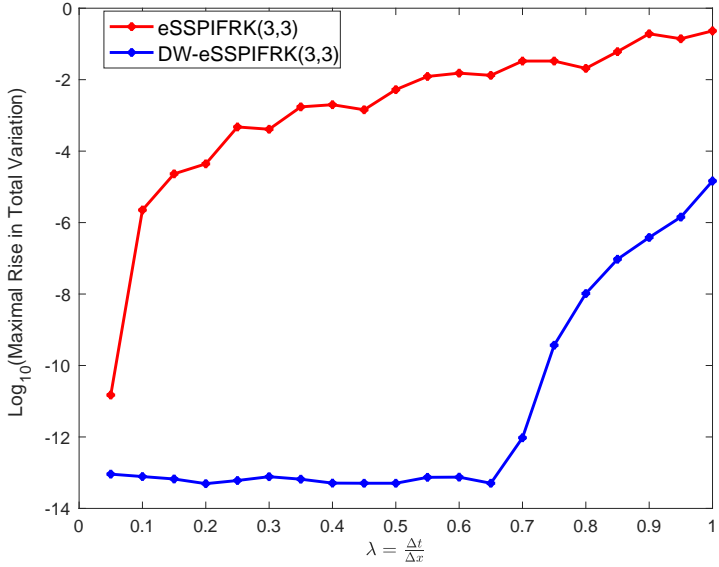} 
  \end{minipage}\hfill
  \begin{minipage}[c]{0.4\textwidth} 
    \caption{ Total variation behavior of the evolution over 25 time-steps evolving the 
     integrating factor methods with the  eSSPRK(3,3) Shu-Osher Runge--Kutta method, 
    \eqref{SOIF} (red) and with the corresponding method with downwinding  \eqref{SOIFDW} (blue).
On the x-axis is the value of $\lambda = \frac{\dt}{\dx}$,
on the y-axis is $log_{10}$ of the maximal rise in TV.}
 \label{fig:motivating}
  \end{minipage}
\end{figure}

For the nonlinear terms, we use  a fifth order WENO finite difference method  \cite{WENOjiang}
\[ N(u) =   WENO\left( - \frac{1}{2}   u^2 \right) \approx  - \left( \frac{1}{2} u^2 \right)_x.\]
Although the WENO method 
is not guaranteed to preserve the total variation behavior, in practice we observe that WENO seems to satisfy 
\[ \| u^n +  \dt    N( u) \|_{TV} \leq \| u^n \|_{TV} \; \; \; \; \mbox{for} \; \; \; \dt \leq \frac{1}{2} \dx \] for this
problem.

  For the time discretization, we use the integrating factor method based on the 
explicit eSSPRK(3,3) Shu-Osher method \eqref{SSPRK33}:
\begin{eqnarray} \label{SOIF}
     u^{(1)} &= & e^{L\dt}u^n + e^{L\dt} \dt N(u^n) \nonumber \\
     u^{(2)} &= & \frac{3}{4} e^{\frac{1}{2}L\dt} u^n + \frac{1}{4} e^{-\frac{1}{2}L\dt} \left( u^{(1)} + \dt N(u^{(1)}) \right) \nonumber \\
     u^{n+1} & = & \frac{1}{3}e^{L\dt}  u^n +\frac{2}{3} e^{\frac{1}{2}L\dt}  \left( u^{(2)} + \dt N(u^{(2)}) \right). 
\end{eqnarray}
The appearance of  negative exponents is due to the fact that the optimal explicit eSSPRK(3,3) Shu-Osher method  \eqref{SSPRK33} has decreasing abscissas.
These terms threaten to destroy the TVD property.

To correct for these negative values, we use the integrating factor method based on the same
explicit eSSPRK(3,3) Shu-Osher method \eqref{SSPRK33}, 
\begin{eqnarray} \label{SOIFDW}
     u^{(1)} &= & e^{L\dt}u^n + e^{L\dt} \dt N(u^n) \nonumber \\
     u^{(2)} &= & \frac{3}{4} e^{\frac{1}{2}L\dt} u^n + \frac{1}{4} e^{-\frac{1}{2}\tilde{L} \dt}  \left( u^{(1)} + 
     \dt N(u^{(1)}) \right)   \nonumber \\
     u^{n+1} & = & \frac{1}{3}e^{L\dt}  u^n +\frac{2}{3} e^{\frac{1}{2}L\dt} \left( u^{(2)} +   \dt N(u^{(2)}) \right). 
\end{eqnarray}
but here, whenever the abscissas are decreasing we use a downwinded operator $\tilde{L} \approx 10 u_x$ defined by 
\begin{eqnarray} \label{downwinding}
  \tilde{L} u =  - a \left(  \frac{u_{j+1} - u_{j} }{\Delta x} \right) .
  \end{eqnarray}
  Note that in this case, $\tilde{L} = - L^T$.
    This operator satisfies the TVD condition 
\[ \| u^n -  \dt    \tilde{L} u \|_{TV} \leq \| u^n \|_{TV}  \; \; \; \; \mbox{for} \; \; \; \dt \leq \frac{1}{a} \dx.\]
  (Again, $a=10$ in our case).

We  selected different values of $\dt$ and used each one to evolve the solution  25 time steps using 
the integrating factor Runge--Kutta  methods \eqref{SOIF} without downwinding and \eqref{SOIFDW} with downwinding.
At each stage we   calculated the maximal rise in total variation  for  25 time steps. 
In Figure \ref{fig:motivating} we show the $log_{10}$ of the maximal rise in total variation vs. the value of $\lambda = \frac{\dt}{\dx}$ of the evolution 
using the standard integrating factor Runge--Kutta method \eqref{SOIF} (in red) and  the method with downwinding \eqref{SOIFDW} (in blue). 
We observe that when downwinding is not used
there is a large maximal rise in total variation even for very small values of $\lambda$. 
However, if we correct for the decreasing abscissas by using
the downwinded operator $\tilde{L}$, as in  \eqref{SOIFDW}, the numerical solution  
maintains a small maximal rise in total variation up to $\lambda \approx  0.65$.

\section{Explicit SSP Runge--Kutta methods} \label{sec:background}

 In this section, we present some popular and efficient explicit SSP Runge--Kutta methods.
SSP Runge--Kutta methods of various stages and order were reported in \cite{SSPbook2011}. The SSP coefficients
 of optimal explicit SSP Runge--Kutta methods of up to $s=10$ stages and
order $p=4$ are in  Table \ref{tab:SSPcoef}. Many of these methods do not feature only non-decreasing abscissas (the second order
methods are an exception). In  Table \ref{tab:SSPcoef+}  we present the corresponding  SSP coefficients of the explicit Runge--Kutta
methods with non-decreasing abscissas.
Unfortunately, no methods of order $p \geq 5$ with positive SSP coefficients can exist \cite{kraaijevanger1991,ruuth2001}.

 \begin{table}
\centering
\setlength\tabcolsep{4pt}
\begin{minipage}{0.48\textwidth}
\centering
\begin{tabular}{|c|lll|} \hline
  \diagbox{s}{p}  &  2 & 3 & 4 \\
 \hline
    1  &  -           &  -           &   -\\   
    2  &  1.0000 &  -           &   -\\    
    3  &  2.0000 &  {1.0000} &   -\\
    4  &  3.0000 &  {2.0000}  &   -\\
    5  &  4.0000 &  2.6506 &   {1.5082}\\
    6  &  5.0000 &  3.5184 &   2.2945\\
    7  &  6.0000 &  4.2879 &   3.3209\\
    8  &  7.0000 &  5.1071 &   4.1459\\
   9  &  8.0000 &  6.0000 &   4.9142\\
   10 &  9.0000 &  6.7853 &   {6.0000}\\
\hline
\end{tabular}
\caption{SSP coefficients of the optimal eSSPRK(s,p) methods.}
\label{tab:SSPcoef} 
\end{minipage}%
\hfill
\begin{minipage}{0.48\textwidth}
\centering
\begin{tabular}{|c|lll|} \hline
  \diagbox{s}{p}  &  2 & 3 & 4 \\
 \hline
    1  &  -           &  -           &   -\\   
    2  &  1.0000 &  -           &   -\\    
    3  &  2.0000 &  {0.7500} &   -\\
    4  &  3.0000 &  {1.8182} &   -\\
5  &  4.0000 &  2.6351 &   {1.3466}\\
    6  &  5.0000 &  3.5184 &   2.2738\\
    7  &  6.0000 &  4.2857 &   3.0404\\
    8  &  7.0000 &  5.1071 &   3.8926\\
   9  &  8.0000 &  6.0000 &   4.6048\\
   10 &  9.0000 &  6.7853 &   {5.2997} \\
\hline
\end{tabular}
\caption{SSP coefficients of the optimal eSSPRK+(s,p) methods with non-decreasing abscissas.}
 \label{tab:SSPcoef+} 
\end{minipage}
\end{table}

%In \cite{shu1988b, shu1988} Shu presented a three-stage third order 
%explicit SSP Runge--Kutta methods with  SSP coefficient $\sspcoef=1$,
%that was proven optimal \cite{gottliebshu1998}. This method 
%has been extensively used and is known as the Shu-Osher method. This method
%was used above in ? and has decreasing abscissas:
%\noindent{\bf eSSPRK(3,3):}
%\begin{eqnarray} \label{SSPRK33}
%     u^{(1)} &= & u^n + \dt F(u^n) \nonumber \\
%     u^{(2)} &= & \frac{3}{4} u^n + \frac{1}{4} \left(u^{(1)} + \dt F(u^{(1)})\right)  \nonumber \\
%     u^{n+1} & = & \frac{1}{3} u^n + \frac{2}{3} \left(u^{(2)} +  \dt F(u^{(2)})\right).
%\end{eqnarray}
%
We use the notation eSSPRK(s,p) to denote an explicit SSP Runge--Kutta method with
$s$ stages and of order $p$. 
As in \cite{SSPIFRK2018} we use the notation  eSSPRK+(s,p) to denote the corresponding method with 
non-decreasing abscissas. 
In this paper we consider the Shu-Osher method eSSPRK(3,3) 
as well as the eSSPRK(4,3), eSSPRK(5,4), and eSSPRK(10,4). We selected these methods  by 
examining the SSP coefficients in Tables \ref{tab:SSPcoef} and \ref{tab:SSPcoef+} above and selecting 
the two methods for third order and fourth order for which the SSP coefficient of the eSSPRK+(s,p) method is 
significantly smaller than the corresponding eSSPRK(s,p) method. In fact, these are good methods to explore as
the eSSPRK(3,3) and eSSPRK(10,4) are widely used methods. These methods are given below:

\noindent{\bf eSSPRK(3,3):}
\begin{eqnarray} \label{SSPRK33}
     u^{(1)} &= & u^n +  \dt F(u^n) \nonumber \\
     u^{(2)} &= & \frac{3}{4} u^n + \frac{1}{4}  u^{(1)} + \frac{1}{4}  \dt F(u^{(1)})  \nonumber \\
     u^{n+1} & = & \frac{1}{3}   u^n +\frac{2}{3}   u^{(2)} +  \frac{2}{3}  \dt F(u^{(2)}). 
\end{eqnarray}
This method has $\sspcoef=1$. The abscissas are  \[  (c_1, c_2, c_3)= (0, 1, 1/2).\]

\noindent{\bf eSSPRK(4,3):}
\begin{eqnarray} \label{SSPRK43}
u^{(1)} &= & u^n + \frac{1}{2}\dt F(u^n)  \nonumber \\
u^{(2)} &= & u^{(1)}+ \frac{1}{2}\dt F(u^{(1)})   \nonumber \\
u^{(3)} &= & \frac{2}{3}u^n + \frac{1}{3}\left(u^{(2)}+ \frac{1}{2}\dt F(u^{(2)})  \right)  \nonumber \\
u^{n+1} & = & u^{(3)}+ \frac{1}{2}\dt F(u^{(3)}) 
\end{eqnarray}
This method has $\sspcoef=2$. The abscissas are  \[  (c_1, c_2, c_3, c_4)= (0, 1/2, 1, 1/2).\]

No four stage fourth order explicit Runge--Kutta methods exist with a positive SSP coefficient
 \cite{gottliebshu1998,ruuth2001}. However, fourth order methods with more than four stages
($s>p$) do exist. A five stage fourth order method found by Spiteri and Ruuth \cite{SpiteriRuuth2002}  is\\
 \noindent{\bf eSSPRK(5,4):}
   \begin{eqnarray*}
u^{(1)} & = & u^n +  0.391752226571890 \dt F(u^n) \\ 
u^{(2)} & = &  0.444370493651235 u^n +  0.555629506348765 u^{(1)} 
+ 0.368410593050371 \dt F(u^{(1)}) \\ 
u^{(3)} & = &  0.620101851488403 u^n +  0.379898148511597 u^{(2)} 
 + 0.251891774271694  \dt F(u^{(2)}) \\ 
u^{(4)} & = &  0.178079954393132 u^n + 0.821920045606868 u^{(3)} 
+  0.544974750228521 \dt F(u^{(3)})\\ 
u^{n+1} & = &    0.517231671970585 u^{(2)} 
 +  0.096059710526147 u^{(3)} +  0.063692468666290 \dt F(u^{(3)}) \\ 
& & +  0.386708617503268 u^{(4)} +   0.226007483236906 \dt F(u^{(4)}) \, ,
\end{eqnarray*}
The abscissas are \[  (c_1, c_2, c_3, c_4, c_5) = 
(0,   0.391752226571889,   0.586079689311541,   0.474542363121399,  0.935010630967652  ).\]

A notable example of a fourth order methods with more than four stages  is Ketcheson's eSSPRK(10,4)  that has
$\sspcoef=6 $  and an attractive low storage formulation \cite{ketcheson2008}: \\
  \noindent{\bf eSSPRK(10,4):} 
\begin{eqnarray*}
u^{(1)} & = & u^n + \frac{1}{6} \dt F(u^n) \\ 
u^{(i+1)} & = & u^{(i)} + \frac{1}{6} \dt F(u^{(i)}) \; \; \;  i=1,2, 3\\ 
u^{(5)} & = & \frac{3}{5} u^n +   \frac{2}{5} \left(u^{(4)} +  \frac{1}{6} \dt F(u^{(4)})\right) \\ 
u^{(i+1)} & = & u^{(i)} + \frac{1}{6} \dt F(u^{(i)}) \; \; \; i=5,6,7,8 \\ 
u^{n+1} & = &   \frac{1}{25} u^{n} +  \frac{9}{25} \left(u^{(4)} + \frac{1}{6}  \dt F(u^{(4)}) \right)+  \frac{3}{5} \left(u^{(9)}
+ \frac{1}{6}  \dt F(u^{(9)})\right)  \, ,
\end{eqnarray*}
has  $\sspcoef=6$.
The abscissas are 
 \[  (c_1, c_2, c_3, c_4, c_5, c_6, c_7, c_8, c_9, c_{10}) =
( 0,    1/6,     1/3,      1/2,  2/3,    1/3,    1/2,   2/3,     5/6,   1 ).\]     

These four eSSPRK(s,p) methods have an SSP coefficient that is significantly larger than the corresponding 
methods with only non-decreasing abscissas, eSSPRK+(s,p), as we can see in Tables \ref{tab:SSPcoef} and  
\ref{tab:SSPcoef+}. This leads us to expect that for these $(s,p)$ combinations, using the downwinding operator $\tilde{L}$
to salvage the SSP property of the eSSPRK methods would be more efficient than using the corresponding 
eSSPRK+ method with only non-decreasing coefficients. In the following section, we use these methods in numerical tests and compare their
performance.

\section{Numerical Results\label{sec:test}}

\subsection{Sharpness of SSP  time-step}
As in the motivating example, we consider Burgers' equation with a linear advection term
\begin{align} 
U_t + 10  U_x +  \left( \frac{1}{2} U^2 \right)_x & = 0 \hspace{.75in}
    U(0,x)  =
\begin{cases}
1, & \text{if } 1/4 \leq x \leq 3/4 \\
0, & \text{else }    \nonumber
\end{cases}
\end{align}
on the domain $x \in [0,1]$ with periodic boundary conditions. 
 We discretize the spatial grid with $1000$ points and use
a first-order upwind difference  $L u \approx -10 u_x  $ defined by 
\eqref{upwinding}
to semi-discretize the linear term. As mentioned above, this operator satisfies the TVD condition 
\[ \| u^n +  \dt    {L} u \|_{TV} \leq \| u^n \|_{TV}  \; \; \; \; \mbox{for} \; \; \; \dt \leq \frac{1}{10} \dx.\] 
When the abscissas decrease, the downwind operator is used instead, as in \eqref{rkIFSO}.
This downwind operator is defined by \eqref{downwinding} and satisfies the TVD condition 
\[ \| u^n -  \dt    \tilde{L} u \|_{TV} \leq \| u^n \|_{TV}  \; \; \; \; \mbox{for} \; \; \; \dt \leq \frac{1}{10} \dx.\] 
\begin{figure}
\includegraphics[scale=.45]{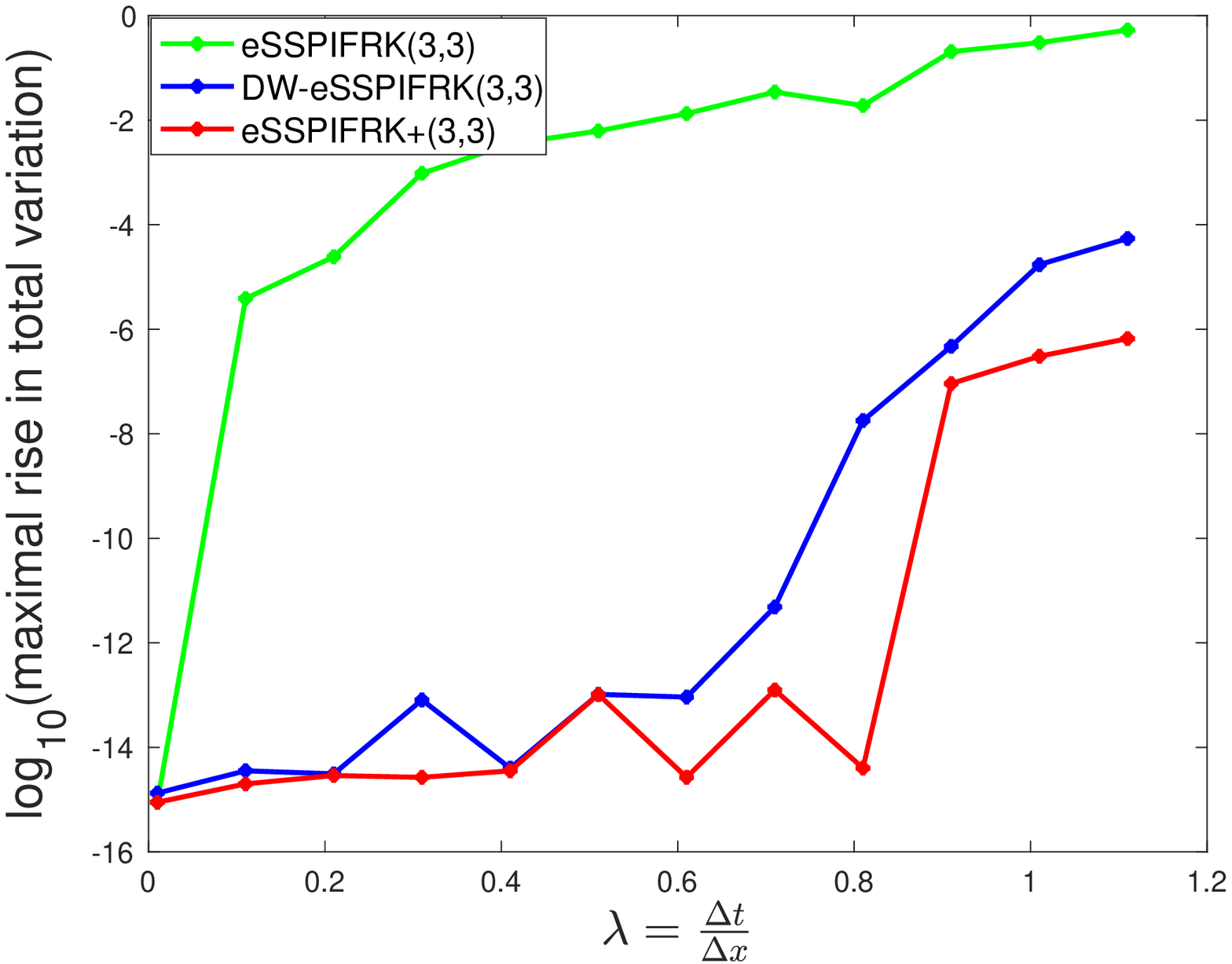} 
\includegraphics[scale=.45]{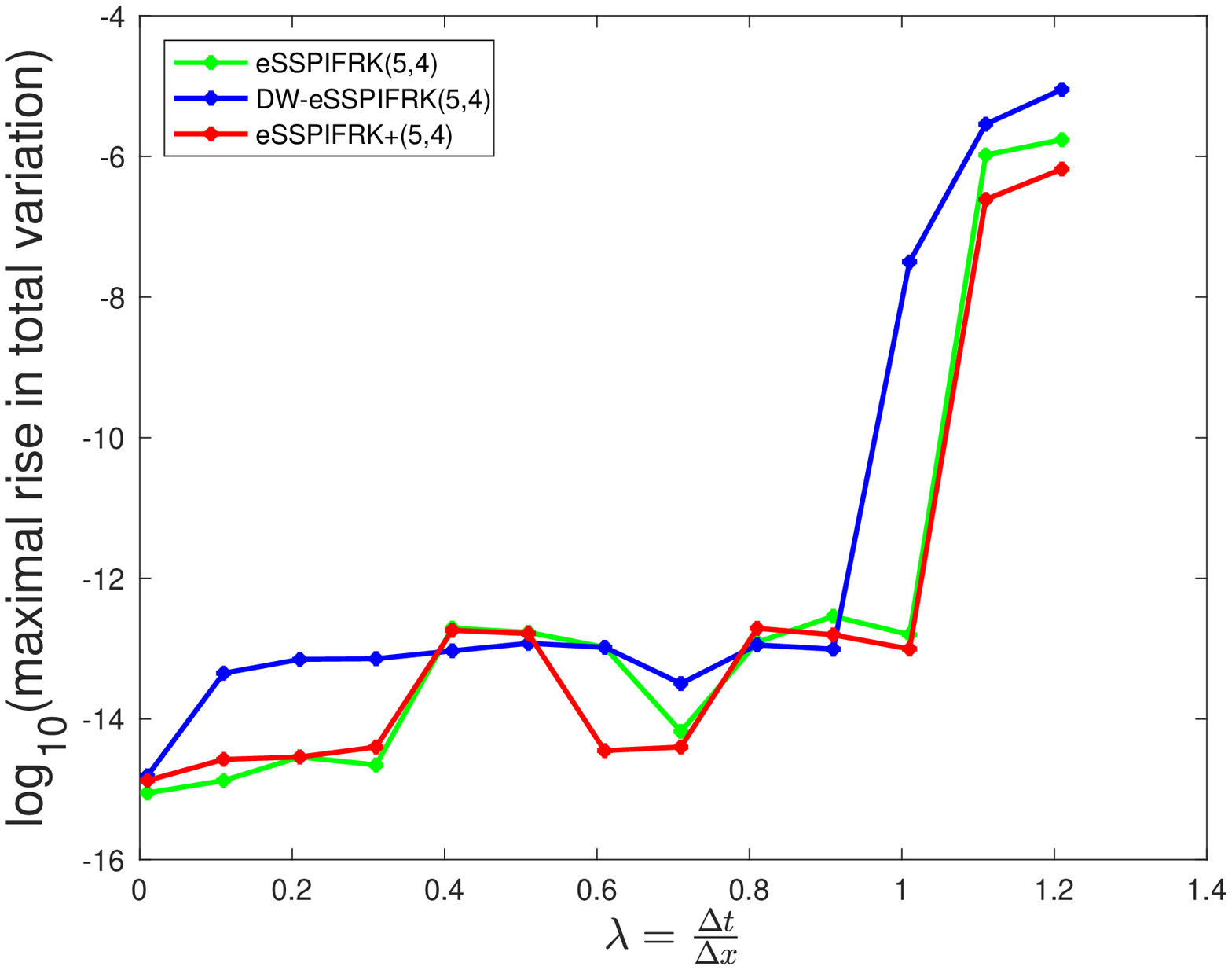} 
\\
\includegraphics[scale=.45]{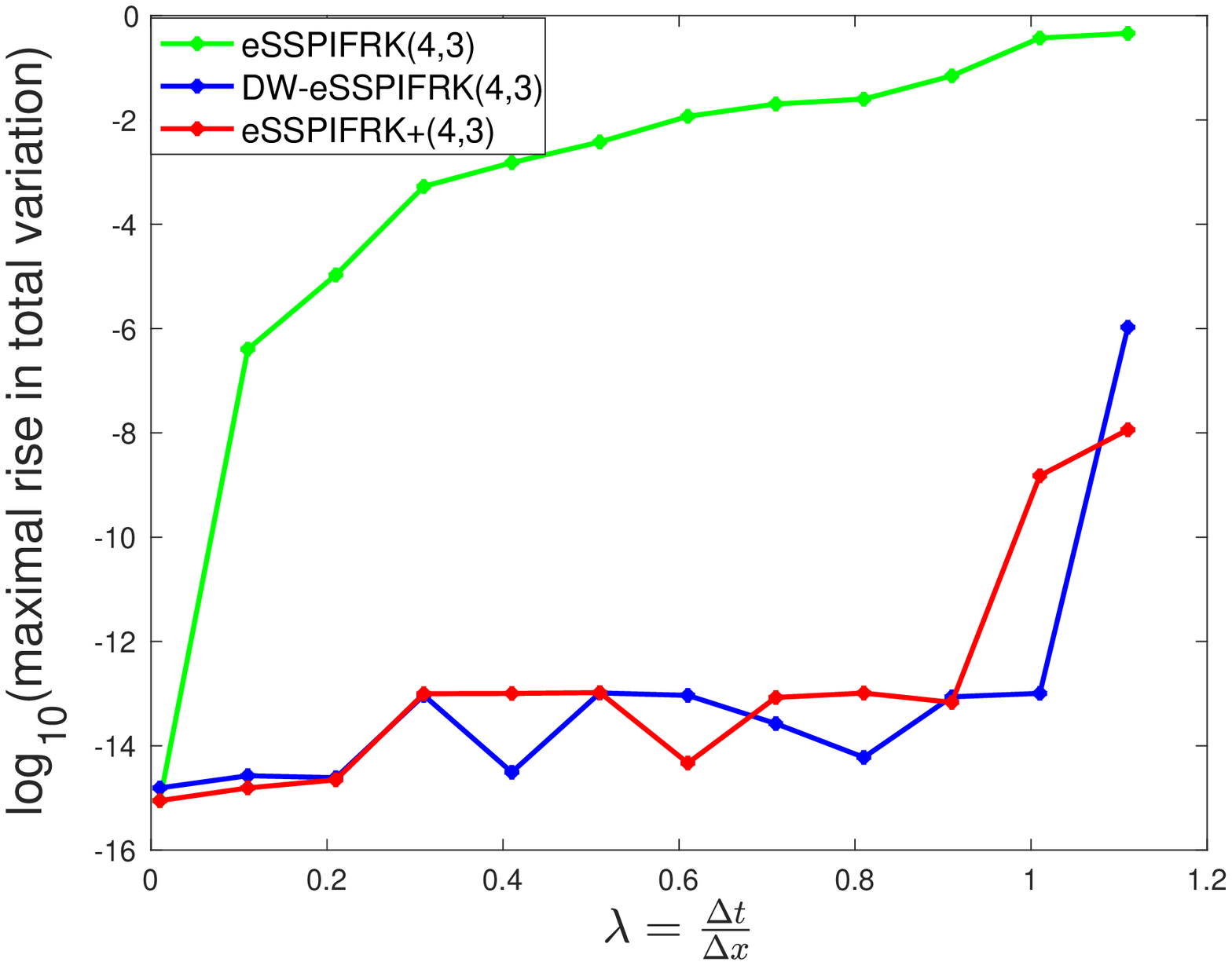} 
\includegraphics[scale=.45]{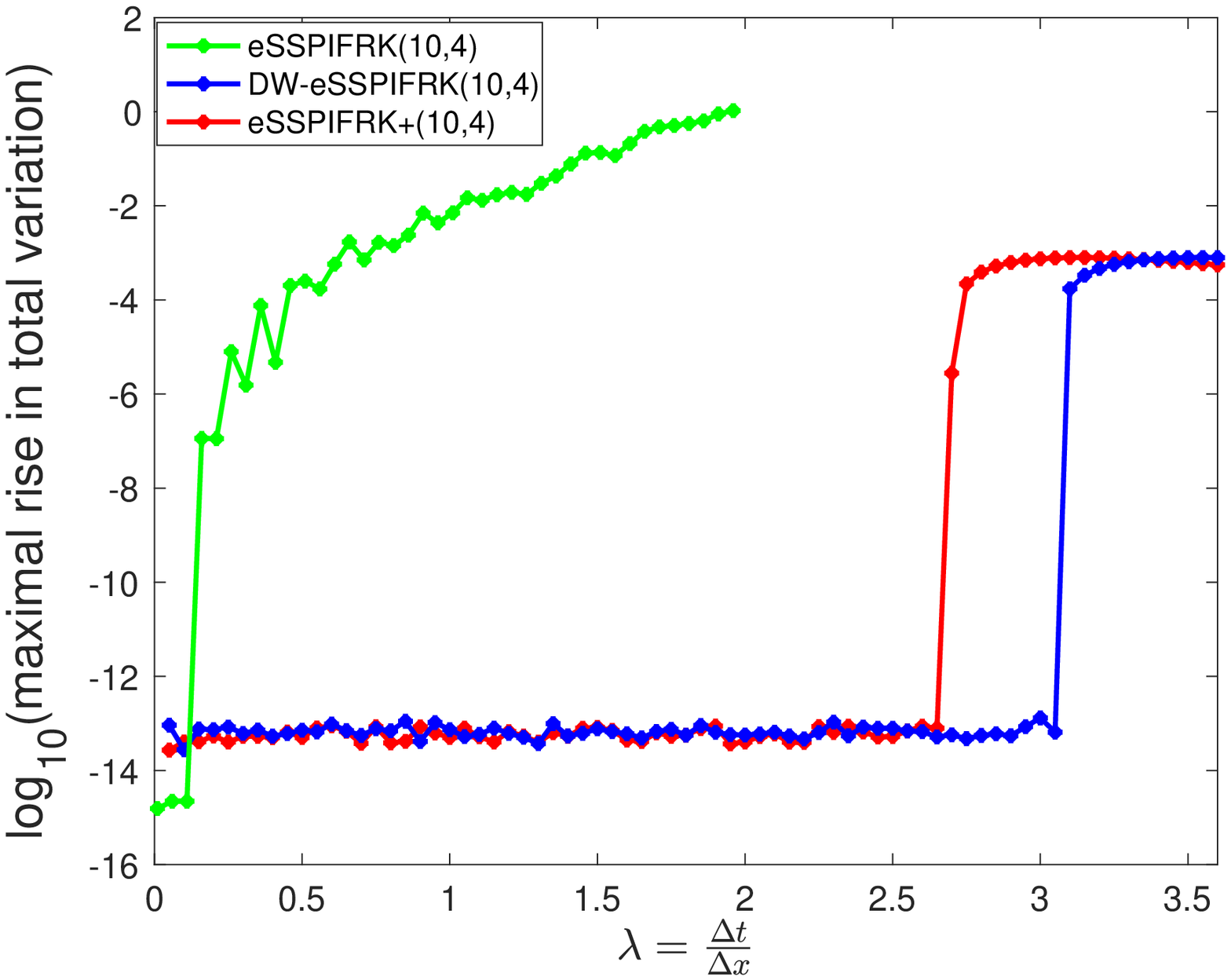}
    \caption{ Total variation behavior of the evolution over 25 time-steps evolving the 
     integrating factor methods with the  eSSPRK(s,p) Shu-Osher Runge--Kutta method, 
(green) and with the corresponding method with downwinding   (blue), as well as a comparison
with the eSSPRK+(s,p) method with non-decreasing abscissas (red). The methods selected
are $(s,p)=(3,3),(4,3),(5,4), (10,4)$.
On the x-axis is the value of $\lambda = \frac{\dt}{\dx}$,
on the y-axis is $log_{10}$ of the maximal rise in TV.}
 \label{fig:TVDtest}
\end{figure}

For the nonlinear terms, we use  a fifth order WENO finite difference method  \cite{WENOjiang}
\[ N(u) = \frac{1}{2} WENO\left( - u^2 \right) \approx  - \left( \frac{1}{2} u^2 \right)_x\]. Although the WENO method 
is not guaranteed to preserve the total variation behavior, in practice we observe that WENO seems to satisfy 
\[ \| u^n +  \dt    N( u) \|_{TV} \leq \| u^n \|_{TV} \; \; \; \; \mbox{for} \; \; \; \dt \leq \frac{1}{2} \dx \] for this
problem.

We measure the total variation of the numerical solution at each stage, and compare it to the 
total variation at the previous stage. We are interested in the size of  time-step $\dt$ at which the total variation
begins to rise. We refer to this value as the {\em observed TVD time-step}. We are interested in 
comparing this value with the expected TVD time-step dictated by the theory. We call the SSP coefficient
corresponding to the value of the observed TVD time-step the {\em observed SSP coefficient} $\sspcoef_{obs}$.
 In Figure  \ref{fig:TVDtest} we show the  $log_{10}$ of the maximal rise in total variation
versus the ratio $\lambda = \frac{\dt}{\Delta x}$, for methods with $(s,p)=(3,3), (4,3), (5,4)$, and $(10,4)$.
In each graph, we compare the integrating factor Runge--Kutta  using the eSSPRK(s,p) method with and without downwinding,
to the  integrating factor Runge--Kutta  using the eSSPRK+(s,p) method. 

The green lines in  Figure \ref{fig:TVDtest} show that if we do not correct for the decreasing abscissas, 
the total variation is usually not well-controlled. This is true for the 
methods with $(s,p)=(3,3), (4,3)$, and $(10,4)$.
However,  the eSSPRK(5,4) method works well without downwinding, and in fact its performance is identical to that of
the  eSSPRK+(5,4) method. On the other hand, downwinding negatively impacts the size of the time-step 
for which the total variation begins to rise. Although the integrating factor approach with  eSSPRK(5,4)  and downwinding
behaves even better than predicted by the theory, the  eSSPRK+(5,4) method  
out-performs the theoretical bound by more  \cite{SSPIFRK2018} .
This highlights the fact that while downwinding guarantees that 
the strong stability property will be preserved when the abscissas decrease, the lack of this guarantee does not
always mean that the strong stability property will be violated.

Comparing the blue and red lines in Figure  \ref{fig:TVDtest}  we note that for the $(s,p)=(3,3)$ method, 
the eSSPRK+(3,3) method with non-decreasing abscissas (red) out-performs the eSSPRK+(3,3) method with downwinding (blue). 
This may be explained by the fact that this methods performed better than expected: the solution was TVD for larger time-step than predicted
by the SSP coefficient (see \cite{SSPIFRK2018} for a discussion of this behavior). 

For the methods with $(s,p)=(4,3)$ and $(s,p)=(10,4)$ 
we observe in practice the behavior predicted by the theory: 
when the eSSPRK(s,p) method is not corrected with downwinding (green line), 
we observe a rise in total variation for any value of $\lambda$. 
When the integrating factor   eSSPRK(s,p) method is  
corrected with downwinding when the  abscissas decrease (blue) the allowable time-step for SSP is larger
than for the  eSSPIFRK+(s,p) method (red)  as predicted by the  theory. 
As we noted in \cite{SSPIFRK2018},
the SSP coefficients of the eSSPRK+(s,p) methods are approximately 10\% smaller than  those of the eSSPRK(s,p) methods,
so the advantage of using downwind over using a method with non-decreasing abscissas is relatively modest.
In the case of $(s,p)=(4,3)$ the increase is from $\lambda_{TV}=0.9$ for the method with non-decreasing abscissas to $\lambda_{TV}=1$  
for the methods with downwinding, and for $(s,p)=(10,4)$ the increase is from $\lambda_{TV}=2.65$ to $\lambda_{TV}=3$.

\subsection{Accuracy studies}

Consider the problem
\begin{align} \label{PDEtest}
U_t +  U_x +   \left( \frac{1}{2} U^2 \right)_x & = 0 \hspace{.75in}
    U(0,x)  = \frac{1}{2} \left( 1 + \sin( x) \right)
\end{align}
on the domain $ 0 \leq x \leq 2 \pi$.
We  use the fifth order WENO for the nonlinear term, and upwind finite difference methods 
%first and second order finite difference, as well as a spectral  difference 
to spatially discretize the linear advection term. 
For the time-discretization we step to final time $T_f = 1.0$ using 
an integrating factor Runge--Kutta approach with eSSPRK(3,3) and  eSSPRK(10,4) with and without downwinding.
 
To compute the highly accurate reference solution to the PDE we used $24,000$ points in space and 
a spectral differentiation operator for the linear advection  term with a fifth order WENO for the nonlinear Burgers' term;
for the time evolution we used   MATLAB's embedded ODE45 routine with $AbsTol=10^{-14}$ and $RelTol=5 \times 10^{-14}$.
In this accuracy test, we have a smooth solution (for the  time interval selected), and so the  spectral differentiation operator, which does
not have the nonlinear stability properties needed for the solution of a problem with discontinuities, will be stable for this 
smooth problem.
 
 \begin{figure}[ht]
\includegraphics[scale=.425]{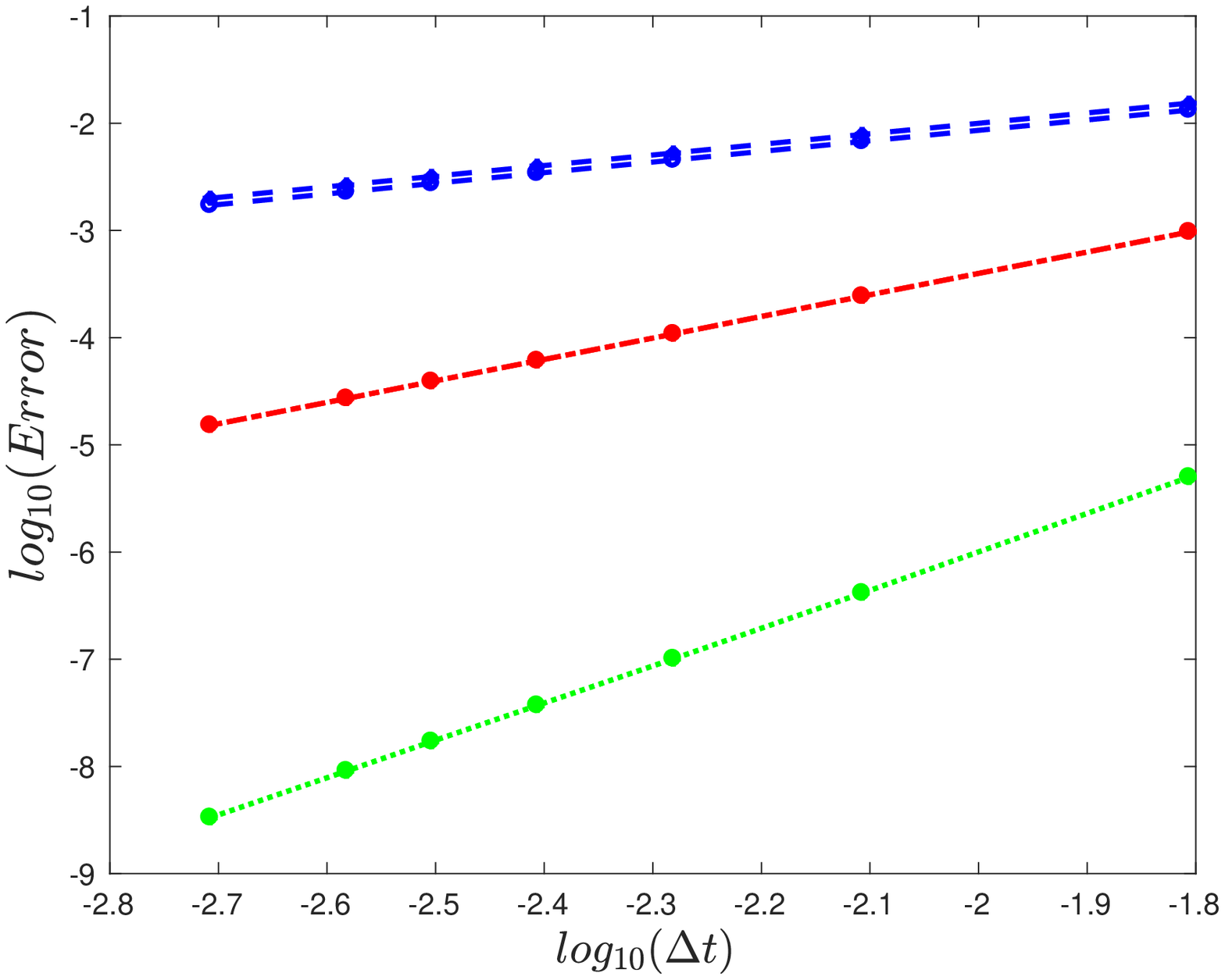} 
\includegraphics[scale=.425]{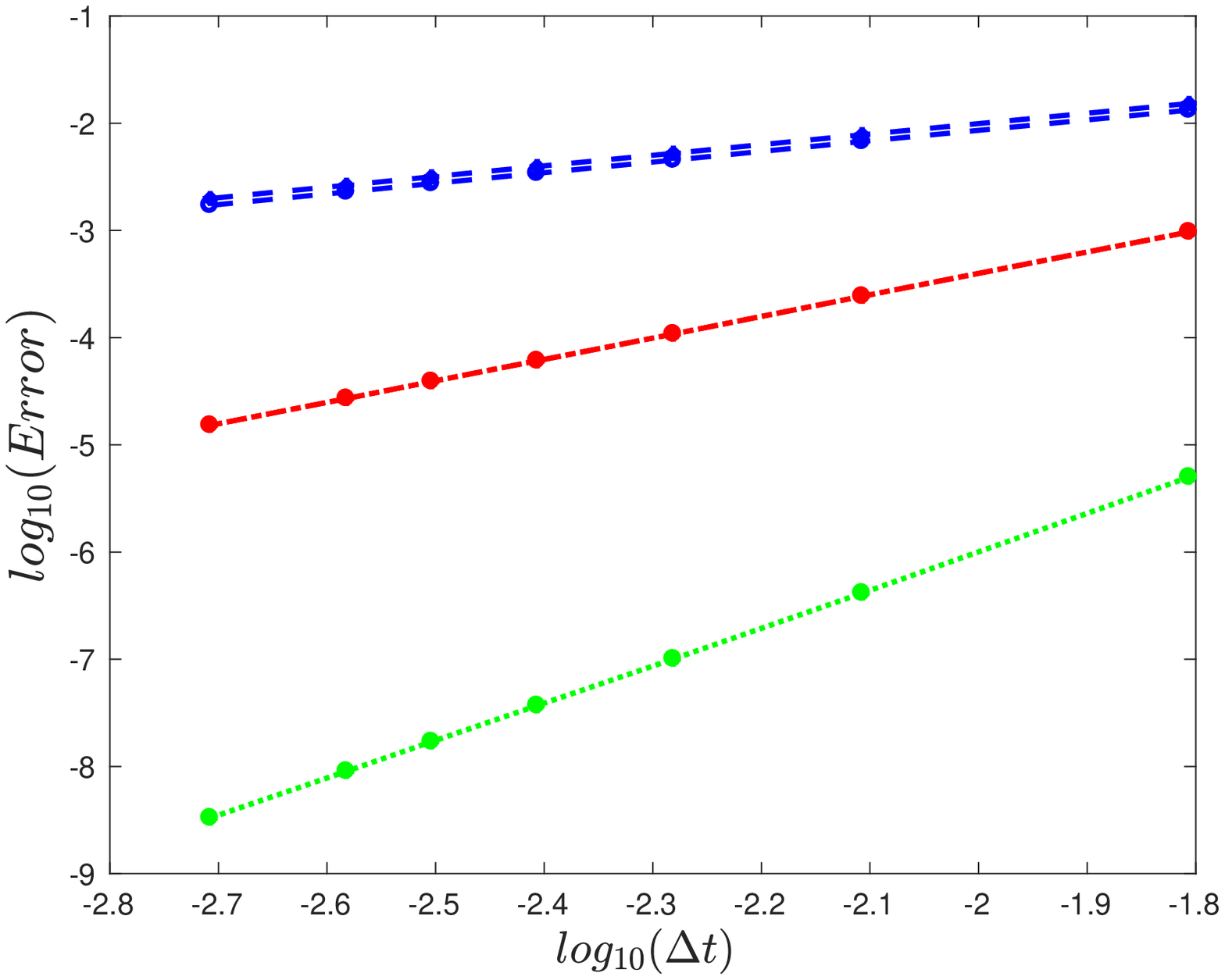} 
    \caption{\small Test 1: Log-log plot of the $L_2$ errors vs. the timestep using the 
     integrating factor Runge--Kutta method eSSPRK(s,p) with downwinding (with * markers) and the 
 integrating factor Runge--Kutta method  with non-decreasing abscissas eSSPRK+(s,p) (with o markers).
 The blue line is for the first order spatial operator $L1_{N_x}$, the red line is for the second order spatial operator $L2_{N_x}$, 
 and the green line for the spectral spatial operator $Lspec_{N_x}$.
 Left: $(s,p)=(3,3)$. Right: $(s,p)=(10,4)$.} Note that the markers look like solid circles because the * markers overlap with the o markers.
 \label{fig:accuracy}
\end{figure}

\vspace*{-.25in}

\noindent{\bf Test 1: space-time co-refinement study.}\\
 In our first test, we use  $N_x= [100,200,300,400,500,600,800]$ points in space and 
a time-step of $\dt = \frac{\dx}{4}$. For the spatial discretization we use a first order upwind operator $L1_{N_x}$, the 
second order upwind operator $L2_{N_x}$, and the spectral operator $Lspec_{N_x}$.
For each value of $N_x$ we compute the error vector and calculate its $L_2$ norm. In Figure \ref{fig:accuracy} we compare the 
convergence of the integrating factor Runge--Kutta method eSSPRK(3,3) with downwinding to the 
 integrating factor Runge--Kutta method  with non-decreasing abscissas eSSPRK+(3,3).  
 We observe that when using a low-order spatial discretization $L1_{N_x}$ and $L2_{N_x}$ for the linear advection term
 the spatial error is clearly dominating, and the order of convergence is first and second order respectively.
 The results are the same when we use eSSPRK(10,4) with downwinding and eSSPRK+(10,4).
 When using the spectral discretization $Lspec_{N_x}$ for the linear advection term, we see convergence of order $3.5$.
 This order is the same for all the methods, whether using the third or fourth order time discretization, so we conclude
 that here, too, the spatial error dominated.

This first test shows that the integrating factor approach using the optimal SSP Runge--Kutta method while
incorporating downwinding converges properly, and its errors are close to
identical to the non-decreasing abscissa approach described in \cite{SSPIFRK2018}.
This establishes that, as expected, downwinding is an appropriate technique to employ when dealing with a PDE.

 \begin{figure}[htb]
\includegraphics[scale=.425]{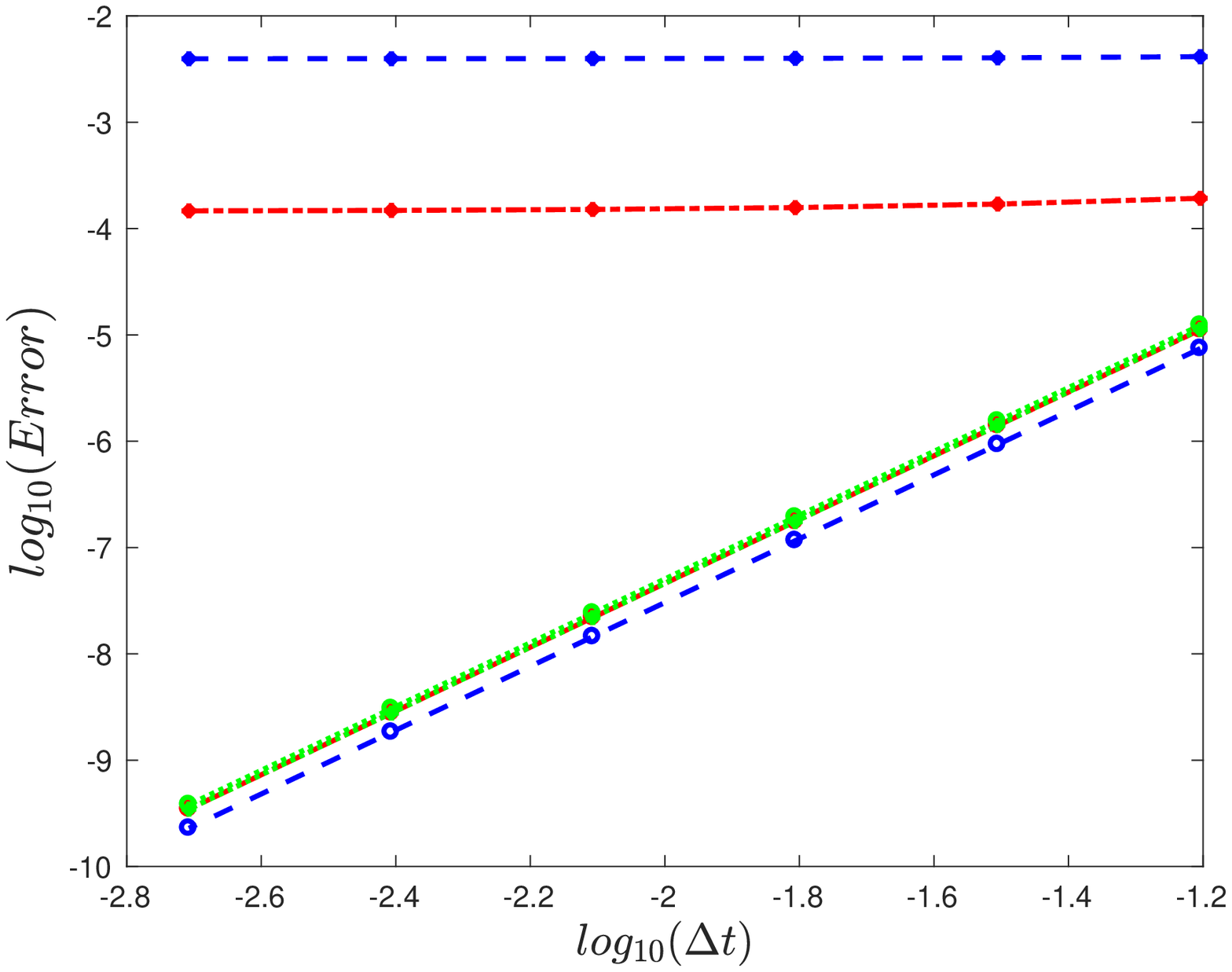} 
\includegraphics[scale=.425]{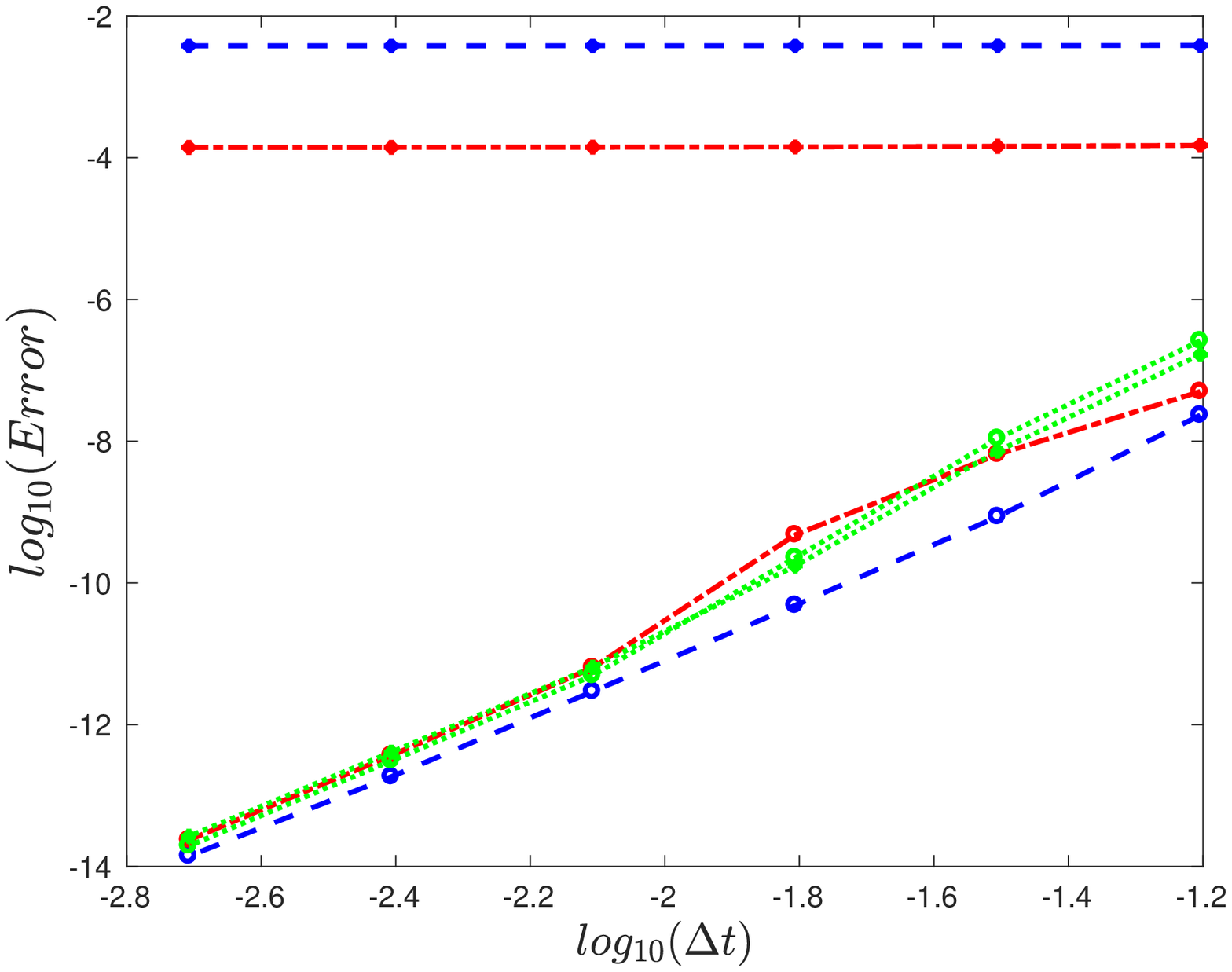} 
    \caption{\small  Test 2: log-log plot of the $L_2$ errors vs. the timestep using the 
     integrating factor Runge--Kutta method eSSPRK(s,p) with downwinding (with * markers) and the 
 integrating factor Runge--Kutta method  with non-decreasing abscissas eSSPRK+(s,p) (with o markers).
 The blue line is for the first order spatial operator $L1_{50}$, the red line is for the second order spatial operator $L2_{50}$, 
 and the green line for the spectral spatial operator $Lspec_{50}$.
 Left: $(s,p)=(3,3)$. Right: $(s,p)=(10,4)$. 
 Note that the markers look like solid circles because the * markers overlap with the o markers.
  \label{fig:accuracy2}}
\end{figure}

It is important to keep in mind that while the PDE is being approximated equally well whether downwinding is 
employed or not, this is not the case for convergence to the ODE. This is due to the fact that while the ODE
resulting from discretizing \eqref{PDEtest} with $L$ is 
\begin{align} \label{ODEtest1}
u_t  = L u + N(u)  
\end{align}
the ODE that results from discretizing \eqref{PDEtest} with $\tilde{L}$ is 
\begin{align} \label{ODEtest2}
u_t  = \tilde{L} u + N(u)  ,
\end{align}
which is a different ODE, with a correspondingly different solution.
%To see how this impacts the integrating factor method, consider the simple Runge--Kutta method
% \begin{eqnarray}
%y^{(1)} & = & y^n + \frac{3}{2} \dt F(y^n) \\
% y^{n+1} & = & \frac{5}{9} y^n + \frac{2}{9} \left( y^n+ \frac{3}{2} \dt F(y^n) \right) 
%+  \frac{2}{9} \left( y^{(1)}+ \frac{3}{2} \dt F(y^{(1)}) \right) . \\
% \end{eqnarray}
% When used with the integrating factor approach with downwinding, this method becomes
%  \begin{eqnarray}
% u^{(1)} & = & e^{\frac{3}{2} \dt L} \left( u^n + \frac{3}{2} \dt N(u^n)  \right) \\
%  u^{n+1} & = &   \frac{5}{9}  e^{\dt L} u^n + \frac{2}{9} e^{\dt L}  \left( u^n+ \frac{3}{2} \dt N(u^n) \right) 
%+  \frac{2}{9} e^{-\frac{1}{2} \dt L}   \left( u^{(1)}+ \frac{3}{2} \dt N(u^{(1)}) \right)  
% \end{eqnarray}
%the final line is, equivalently:
%\[ u^{n+1} = \frac{1}{9} e^{\dt L} \left( 7 + 2 e^{\frac{1}{2} \dt (L-\tilde{L} )} \right) u^n
%+ \frac{1}{3} e^{\dt L} \left( 1 +  e^{\frac{1}{2} \dt (L-\tilde{L} )} \right) N(u^n) +
%+ \frac{1}{3} e^{-  \frac{1}{2} \dt L} \left( e^{\frac{1}{2} \dt (L-\tilde{L} )} \right) N(u^{(1)}) .
%\]
%Each term has a $e^{\frac{1}{2} \dt (L-\tilde{L} )} \approx 1$, since $L$ and $\tilde{L}$ approximate the same
%partial derivative to some order. If $L = \tilde{L}$ we recover the original expression
%\[ u^{n+1} =e^{\dt L} u^n + \frac{2}{3} e^{\dt L} N(u^n) +  \frac{1}{3} e^{-  \frac{1}{2} \dt L} N(u^{(1)}) .\]
%
We perform the following numerical test to see how downwinding impacts the solution to the ODE.

\noindent{\bf Test 2: ODE convergence study.} \\
We chose $N_x=50$ points and discretize the linear advection term using 
$L1_{50}$, $L2_{50}$, and $Lspec_{50}$. The nonlinear term is computed as above using WENO, but with 
$N_x=50$ points. 
The reference solution is found by evolving the ODE resulting from this semi-discretization 
using MATLAB's embedded ODE45 routine with $AbsTol=10^{-14}$ and $RelTol=5 \times 10^{-14}$.

 We evolve this semi-discrete ODE using the integrating factor approach based on the  eSSPRK(s,p) 
 methods with downwinding for the cases when the abscissas decrease, and using the integrating factor
 Runge--Kutta methods based on the eSSPRK+(s,p) methods, for 
 $\dt = \lambda \dx $ where $\lambda = \frac{1}{2}, \frac{1}{4}, \frac{1}{8},
\frac{1}{16}, \frac{1}{32}, \frac{1}{64}.$
For each value of $\dt$ we compute the error vector and calculate its $L_2$ norm. We compare the 
convergence of the integrating factor Runge--Kutta method eSSPRK(s,p) with downwinding to the 
 integrating factor Runge--Kutta method  with non-decreasing abscissas eSSPRK+(s,p),
 for $(s,p)=(3,3), (10,4)$.
 In Figure \ref{fig:accuracy2} we observe that when the number of points in space is fixed
 and only the time-step is refined, the   integrating factor  Runge--Kutta method with downwinding  
 has a large error which  the solution  hangs at about $3.8 \times 10^{-3}$ (blue).
 We repeat this test with a second order upwind operator  $L2_{50}$ for the linear advection operator and 
 observe that the solution  hangs at about $1.5 \times 10^{-4}$ (red). 
  In contrast, the methods without downwinding exhibit the expected order  of convergence in time, 
  regardless of the order of the spatial operator. 
  For the spectral operator, $L=\tilde{L}$, so as expected, there is no difference when downwinding is used.
  (Note that  the accuracy results when we use  eSSPRK(s,p) without downwinding are similar to those of   eSSPRK+(s,p)).
  This behavior is well-known in problems with downwinding and discussed extensively in \cite{Ketcheson2017}.

\section{Conclusions}  \label{sec:conclusions}
In \cite{SSPIFRK2018} we first considered the strong stability properties of integrating factor Runge--Kutta methods.
In that work we presented sufficient conditions for preservation of strong stability for integrating factor Runge--Kutta methods,
which required the use of explicit SSP Runge--Kutta methods with non-decreasing abscissas, denoted eSSPRK+ methods. 
When considering methods of order $p=3,4$   many of the eSSPRK+(s,p) methods have smaller SSP coefficients (and therefore smaller allowable time-step) than the optimal eSSPRK(s,p) methods, which often have
some decreasing abscissas.

 In this work, we consider a different approach to preserving the strong stability properties of 
 integrating factor Runge--Kutta methods. In this case, when the abscissas of the eSSPRK methods
 are decreasing, we replace the spatial operator $L$ with a downwinded spatial operator $\tilde{L}$ to
preserve the strong stability properties of integrating factor Runge--Kutta method. 
We presented a complete SSP theory for this approach.
However, our numerical examples show that the downwinded spatial operators introduce some errors that may
adversely affect the accuracy of the methods, and that in most cases the integrating factor approach
with explicit SSP Runge--Kutta methods with non-decreasing abscissas performs nearly as well, 
 if not better, than with explicit SSP Runge--Kutta methods with downwinding.
 These results lead us to conclude that the downwinding approach may not, in practice,
 provide much benefit over using explicit SSP Runge--Kutta methods with non-decreasing abscissas, as described in 
 \cite{SSPIFRK2018}.

{\bf Acknowledgment.} 
This publication is based on work supported by  AFOSR grant FA9550-15-1-0235
and NSF grant DMS-1719698.


\begin{thebibliography}{10}

%\bibitem{EXPINTpackage}
%{\sc H.~Berland, B.~Skaflestad, and W.~M. Wright}, {\em Expint -- a
%  {M}{A}{T}{L}{A}{B} package for exponential integrators}, 2005.
%
%\bibitem{EXPINT}
%{\sc H.~Berland, B.~Skaflestad, and W.~M. Wright}, {\em Expint -- a
%  {M}{A}{T}{L}{A}{B} package for exponential integrators}, ACM Transactions in
%  Mathematical Software, 33 (2007).

\bibitem{IMEX} 
{\sc S. Conde,  S. Gottlieb, Z. Grant,  and J.N. Shadid},
{\em Implicit and Implicit-Explicit Strong Stability Preserving Runge--Kutta Methods with High Linear Order.}
{Journal of Scientific Computing} {  73(2) (2017),} pp. 667-690.

\bibitem{CoxMatthews2002}
{\sc S.~Cox and P.~Matthews}, {\em Exponential time differencing for stiff
  systems}, Journal of Computational Physics, 176 (2002), pp.~430--455.

%\bibitem{SSPIFgithub}
%{\sc S.~Gottlieb, Z.~Grant, and L.~Isherwood}, {\em Optimized strong stability
%  preserving integrating factor {R}unge--{K}utta methods}.
%\newblock \url{https://github.com/SSPmethods/SSPIFRK-methods}.

\bibitem{SSPbook2011}
{\sc S.~Gottlieb, D.~I. Ketcheson, and C.-W. Shu}, {\em Strong Stability
  Preserving Runge--Kutta and Multistep Time Discretizations}, World Scientific
  Press, 2011.

\bibitem{gottliebshu1998}
{\sc S.~Gottlieb and C.-W. Shu}, {\em Total variation diminishing runge--kutta
  methods}, Mathematics of Computation, 67 (1998), pp.~73--85.

\bibitem{gottlieb2001}
{\sc S.~Gottlieb, C.-W. Shu, and E.~Tadmor}, {\em {Strong Stability Preserving
  High-Order Time Discretization Methods}}, SIAM Review, 43 (2001),
  pp.~89--112.
  
  \bibitem{Ketcheson2017}
  {\sc I. Higueras, D. I. Ketcheson, and T. A. Kocsis},
  {\em Optimal monotonicity-preserving perturbations of a given Runge-Kutta method},
  Journal of Scientific Computing (2018). 

\bibitem{SSPIFRK2018}
{\sc L. Isherwood, S. Gottlieb, and Z. Grant},
{\em Strong Stability Preserving Integrating Factor Runge--Kutta Methods.}
arXiv:1708.02595  (2017).

\bibitem{WENOjiang}
{\sc G.-S. Jiang, and C.-W. Shu},
{\em Efficient Implementation of Weighted ENO Schemes},
Journal of Computational Physics, 126  (1996), pp. 202-228.


\bibitem{ketcheson2008}
{\sc D.~I. Ketcheson}, {\em Highly efficient strong stability preserving
  {R}unge--{K}utta methods with low-storage implementations}, SIAM Journal on
  Scientific Computing, 30 (2008), pp.~2113--2136.

\bibitem{kraaijevanger1991}
{\sc J.~F. B.~M. Kraaijevanger}, {\em Contractivity of {R}unge--{K}utta
  methods}, BIT, 31 (1991), pp.~482--528.

\bibitem{lawson1967}
{\sc J. D. Lawson},
{\em Generalized Runge-Kutta Processes for Stable Systems with Large Lipschitz Constants},
SIAM Journal on  Numerical Analysis, 4(3) (1967), 372--380. 

\bibitem{ruuth2001}
{\sc S.~J. Ruuth and R.~J. Spiteri}, {\em Two barriers on
  strong-stability-preserving time discretization methods}, Journal of
  Scientific Computation, 17 (2002), pp.~211--220.

\bibitem{shu1988b}
{\sc C.-W. Shu}, {\em Total-variation diminishing time discretizations}, SIAM
  Journal Scientific Statistical Computing, 9 (1988), pp.~1073--1084.

\bibitem{shu1988}
{\sc C.-W. Shu and S.~Osher}, {\em Efficient implementation of essentially
  non-oscillatory shock-capturing schemes}, Journal of Computational Physics,
  77 (1988), pp.~439--471.

\bibitem{spijker2007}
{\sc M.~Spijker}, {\em Stepsize conditions for general monotonicity in
  numerical initial value problems}, SIAM Journal on Numerical Analysis, 45
  (2008), pp.~1226--1245.

\bibitem{SpiteriRuuth2002}
{\sc R.~J. Spiteri and S.~J. Ruuth}, {\em A new class of optimal high-order
  strong-stability-preserving time discretization methods}, SIAM Journal on
  Numerical Analysis, 40 (2002), pp.~469--491.

\end{thebibliography}
\end{document}